\documentclass[11pt]{amsart}
\usepackage[utf8]{inputenc}
\usepackage{amsthm}
\usepackage{amsmath}
\usepackage{amssymb}
\usepackage{hyperref}
\usepackage{amsrefs}
\newtheorem{theorem}{Theorem}[section]

\newtheorem{lemma}[theorem]{Lemma}

\newtheorem{definition}[theorem]{Definition}

\newtheorem{remark}[theorem]{Remark}

\newtheorem{corollary}[theorem]{Corollary}
\newtheorem{proposition}[theorem]{Proposition}

\numberwithin{equation}{section}

\newcommand{\bb}[1]{\mathbb{#1}}
\newcommand{\cl}[1]{\mathcal{#1}}
\begin{document}
\title[Nearly invariant subspaces in the real Hardy space]{Nearly invariant subspaces in the real Hardy space}

\author{Arshad Khan}
\address{Department Of Mathematics\\
         School of Natural Sciences\\
        Shiv Nadar Institution Of Eminence\\
        Gautam Buddha Nagar - 201314\\
         Uttar Pradesh, India}
\email{ak954@snu.edu.in}

\author{Sneh Lata}
\address{Department Of Mathematics\\
         School of Natural Sciences\\
         Shiv Nadar Institution Of Eminence\\
         Gautam Buddha Nagar - 201314\\
         Uttar Pradesh, India}
\email{sneh.lata@snu.edu.in}

\author{ Dinesh Singh}
\address{Centre For Digital Sciences\\
      O. P. Jindal Global University\\
   Sonipat\\
        Haryana 131001, India}
\email{dineshsingh1@gmail.com}

\subjclass{Primary 47B02; Secondary 47B38, 30H10}

\keywords{Real Hardy space, nearly invariant subspaces, invariant subspaces, backward shift operator, inner function, almost invariant subspaces.}

\begin{abstract}
The objective of this article is to study nearly invariant subspaces of the backward shift operator on the real Hardy space. We also investigate nearly invariant subspaces with finite defect, and as a consequence, provide a characterization for almost invariant subspaces of the backward shift operator.
\end{abstract}

\maketitle

\section{Introduction} \label{intro}

Subspaces associated with the action of the shift and the backward shift operators play a fundamental role in operator theory and function spaces. The classical invariant subspace theory of the Hardy space provides one of the central examples of this phenomenon. In particular, Beurling's theorem gives a complete description of non-trivial shift-invariant subspaces of the Hardy space $H^2(\mathbb D)$, showing that every such subspace is of the form $\theta H^2(\mathbb D)$, where $\theta$ is an inner function. This result has been a starting point for numerous developments concerning invariant, nearly invariant, and almost invariant subspaces.

A natural weakening of the backward shift invariance is provided by nearly invariant subspaces. A subspace $\mathcal M$ of $H^2(\mathbb D)$ is said to be nearly invariant under the backward shift operator $T_z^*$ if
$$
T_z^*(\mathcal M\cap zH^2(\mathbb D))\subseteq \mathcal M.
$$
Such subspaces were introduced by Hitt \cite{Hit} in his study of invariant subspaces of $H^2$ over an annulus. Sarason \cite{Sar} advanced the theory by providing a new and elegant proof of Hitt's theorem through the framework of de Branges–Rovnyak spaces \cite{DBR}. 
\begin{theorem}[Hitt's theorem] \cite[Theorem 1]{Sar} \label{ThH} Suppose $\mathcal{M}$ is a non-trivial subspace of $H^2(\mathbb{D})$ which is nearly invariant under the backward shift operator $T^*_z$, and let $g$ be the function of unit norm in $\mathcal{M}$ which is orthogonal to $\mathcal{M}\cap zH^2(\mathbb{D})$ and positive at the origin. Then there exists a $T^*_z$-invariant subspace $\mathcal{N}$ of $H^2(\mathbb{D})$ such that $\mathcal{M}=g\mathcal{N}$ and $\|gf\|=\|f\|$ for all $f \in \mathcal{N}$.
\end{theorem}
In particular, kernels of Toeplitz operators provide important examples of nearly invariant subspaces; this connection was studied by Hayashi \cite{Hay}. These results established nearly invariant subspaces as a natural and important generalization of invariant subspaces.

Since the time \cite{Hay}, \cite{Hit}, and particularly \cite{Sar} appeared, nearly invariant subspaces have established themselves as an important area of research, and they can be deemed to be a proper generalization of the concept of invariant subspaces. For more insights into 
nearly invariant subspaces, one can refer to \cite{CCP,CGP1,CD,CDP1,DS,Era,KLD,KLS1,LP}, and the references therein. 

Nearly invariant subspaces have since been studied in several broader settings: for the backward shift on vector-valued Hardy spaces by Chalendar, Chevrot and Partington \cite{CCP}, and for operators on more general Hilbert spaces of analytic functions by Erard \cite{Era}, by Khan, Lata and Singh \cite{KLD}, and by Liang and Partington \cite{LP}. A closely related notion is that of an almost invariant subspace, one that becomes invariant upon the addition of a finite-dimensional subspace. Chalendar, Gallardo-Guti\'errez and Partington \cite{CGP1} obtained a Beurling-type theorem for almost invariant subspaces of the backward shift operator on $H^2(\mathbb{D})$, and Chattopadhyay, Das and Pradhan \cite{CDP1} extended this to vector-valued Hardy spaces. These works reflect a continuing interest in weakened forms of shift and backward shift invariance and in their structural description.

While the theory described above is predominantly formulated over the complex field, an important and independently studied setting is provided by the real Hardy space. The real Hardy space
$$
H^2_{\mathbb R}(\mathbb D)
=
\left\{
f(z)=\sum_{n=0}^{\infty}a_nz^n\in H^2(\mathbb D):
a_n\in\mathbb R
\right\}
$$
is a Hilbert space over the field $\mathbb R$. The restriction to real Fourier coefficients imposes a genuine real structure on the space and on the operators acting on it. In particular, the class of invariant subspaces, inner functions, extremal functions, and factorizations are constrained by this real structure.

The study of real Hardy and real Lebesgue spaces has a substantial history
of its own. Mehanna and Singh \cite{MS} described the invariant subspaces of real
$H^1$, Jupiter and Redett \cite{JR} studied invariant subspaces of real $L^1$,
and Raghupathi and Singh \cite{RS} developed function theory in real $H^p$ spaces
and studied invariant subspaces of real $L^p$ spaces. Mikkola and Sasane
\cite{Tol} established a real version of Tolokonnikov's lemma and studied
inner--outer factorization in the real setting, while Wick \cite{BDW,BDW1} 
investigated stabilization problems in the real disc algebra
$A_{\mathbb{R}}(\mathbb{D})$ and in $H^\infty_{\mathbb{R}}(\mathbb{D})$;
such problems arise naturally in control theory, where the transfer
functions of physical systems have real coefficients. Together, these works
show that real function spaces carry a rich theory of their own and call
for a systematic study of their structural properties.

One structural question, however, remains open in this setting. Invariant
subspaces of real Hardy and Lebesgue spaces are by now well understood \cite{JR, MS, RS}, but nearly invariant subspaces, which have been studied
extensively in complex Hardy spaces and in more general Hilbert spaces of
analytic functions \cite{CCP, CDP1, Era, KLD, LP} , have not yet been considered for real
Hardy spaces. This case does not follow from the complex one. A subspace of
$H^2_{\mathbb{R}}(\mathbb{D})$ is closed under multiplication by real
scalars only, not by $i$, so it is not a subspace of $H^2(\mathbb{D})$ in
the usual sense, and Hitt's theorem cannot be applied to it directly. One
can pass to the complex setting, but the functions and subspaces that the
complex theory then provides are, in general, genuinely complex objects
that do not lie in $H^2_{\mathbb{R}}(\mathbb{D})$. The real structure must
therefore be recovered, and showing that it survives is precisely where the
real case calls for arguments of its own.

The purpose of this paper is to develop a theory of nearly invariant
subspaces of the backward shift operator on the real Hardy space
$H^2_{\mathbb{R}}(\mathbb{D})$. Our first main result provides a real
analogue of Hitt's theorem. More precisely, we show that if $\mathcal{M}$ is
a non-zero nearly $T_z^*$-invariant subspace of
$H^2_{\mathbb{R}}(\mathbb{D})$, then there exists a $T_z^*$-invariant
subspace $\mathcal{N}$ of $H^2_{\mathbb{R}}(\mathbb{D})$ and a distinguished
function $g \in \mathcal{M}$, orthogonal to
$\mathcal{M} \cap zH^2_{\mathbb{R}}(\mathbb{D})$, such that
\[
\mathcal{M} = g\mathcal{N}
\]
together with
\[
\|gf\| = \|f\|, \qquad gf \in \mathcal{M}.
\]
The passage from the complex theory to the real theory is facilitated by the
symmetrisation map
\[
\phi(F) = \frac{F + \widehat{F}}{2},
\]
which maps $H^2(\mathbb{D}, \mathbb{C}^m)$ onto
$H^2_{\mathbb{R}}(\mathbb{D}, \mathbb{C}^m)$. This map provides a natural
connection between a real subspace and its complexification. The map $\phi$
is real-linear and contractive, but not complex-linear. This real--complex
correspondence is one of the principal structural ingredients of our proofs.

We next consider nearly invariant subspaces with finite defect. Finite-defect near invariance is a natural extension of near invariance and is particularly useful for describing subspaces whose failure to be nearly invariant is controlled by a finite dimensional subspace. In the complex setting, Chalendar, Gallardo-Guti$\rm{\acute{e}}$rrez, and Partington \cite{CGP1} developed a structural description of nearly invariant subspaces with finite defect of the backward shift operator, while Chattopadhyay and Das \cite{CD} extended this study in more general Hilbert spaces. Motivated by this theory, we establish the corresponding finite-defect results in $H^2_{\mathbb R}(\mathbb D)$. In the defect-one case, we obtain a representation involving a $T_z^*$-invariant subspace of
$H^2_{\mathbb R}(\mathbb D)\oplus H^2_{\mathbb R}(\mathbb D)$,
and then extend the result to arbitrary finite defect. These representations show explicitly how the finite dimensional defect interacts with the real Hardy space structure.

As a consequence, we obtain a characterization of almost invariant subspaces of the backward shift on $H^2_{\mathbb R}(\mathbb D)$. This extends the structural relationship between nearly invariant subspaces with finite defect and almost invariant subspaces in the real Hardy space. Hence, the results of this paper are not motivated solely by the absence of previous results in the real setting; rather, they form part of a broader program of understanding how the structure of backward shift nearly invariant subspaces changes when the underlying scalar field and function space are changed from complex to real.

 The organization of the paper is as follows. Section \ref{PRE} introduces the real and vector-valued Hardy spaces, the symmetrisation map, and several technical lemmas  (Lemma \ref{Lem1}, Lemma  \ref{Lem2}, and Lemma \ref{Lem3}) concerning complexification and the real structure. In Section \ref{NIS}, we study nearly invariant subspaces of the backward shift on $H^2_{\mathbb R}(\mathbb D)$ and prove a real analogue of Hitt's theorem (Theorem \ref{ThM1}). Section \ref{NISF} is devoted to nearly invariant subspaces with finite defect. We first treat the defect-one case (Theorem \ref{MT2})  and then extend the representation to arbitrary finite defect  (Theorem \ref{MT3}). Finally, we use these results to characterize almost invariant subspaces of the backward shift operator on the real Hardy space (Corollary \ref{CorA}).

\section{Preliminaries} \label{PRE}
Suppose $\mathbb{D}$ denotes the open unit disc in the complex plane $\mathbb{C}$. Recall that, the $\mathbb{C}^m$-valued Hardy space, denoted by $H^2(\mathbb{D},\mathbb{C}^m)$, is defined as 
\begin{eqnarray*} 
 H^2(\mathbb{D},\mathbb{C}^m):=\Bigg\{F(z)=\sum_{n=0}^{\infty}A_nz^n: A_n \in \mathbb{C}^m, \sum_{n=0}^{\infty}\|A_n\|_{\mathbb{C}^m}^2 < \infty \Bigg\}.   
\end{eqnarray*}
It is a Hilbert space over the field $\mathbb{C}$ with respect to the inner product defined by
\begin{eqnarray}\label{e2.1}
\langle{F,H}\rangle=\sum_{n=0}^{\infty}\langle{A_n, B_n}\rangle_{\mathbb{C}^m},
\end{eqnarray}  
where
$F(z)=\sum_{n=0}^{\infty}A_n z^n $ and $H(z)=
 \sum_{n=0}^{\infty}B_n z^n$ are in $H^2(\mathbb{D},\mathbb{C}^m)$, and $\langle \cdot,\cdot \rangle_{\mathbb{C}^m}$  denotes the standard inner product on $\mathbb{C}^m$. For $m=1$, we denote $H^2(\mathbb{D}, \mathbb{C})$ simply by $H^2(\mathbb{D})$.
 
 \begin{definition}
   The $\mathbb{C}^m$-valued real Hardy space, denoted by $H^2_{\mathbb{R}}(\mathbb{D},\mathbb{C}^m)$, consists of all functions in $H^2(\mathbb{D},\mathbb{C}^m)$ having Fourier coefficients from $\mathbb{R}^m$ in their power series representations. That is,
\begin{eqnarray*} 
 H^2_{\mathbb{R}}(\mathbb{D},\mathbb{C}^m):=\Bigg\{F(z)=\sum_{n=0}^{\infty}A_nz^n \in H^2(\mathbb{D},\mathbb{C}^m) : A_n \in \mathbb{R}^m\Bigg\}.   
\end{eqnarray*}
It is a Hilbert space over the field $\mathbb{R}$ with respect to the inner product defined in (\ref{e2.1}), where
$F(z)=\sum_{n=0}^{\infty}A_n z^n $ and $H(z)=
 \sum_{n=0}^{\infty}B_n z^n$ are in $H^2_\mathbb{R}(\mathbb{D},\mathbb{C}^m)$. 
\end{definition}
 
The real Hardy space $H^2_\mathbb{R}(\bb D,\bb C)$ 
is denoted simply as $H^2_\mathbb{R}(\bb D)$. 
If $H^2(\mathbb{D},\mathbb{C}^m)$ is viewed as a vector space over the field $\mathbb{R}$, then it is a Hilbert space with respect to the inner product defined in (\ref{e2.1}).

\begin{definition}
 The \textit{shift operator} on $H^2_{\mathbb{R}}(\mathbb{D},\mathbb{C}^m)$ , denoted by $T_z$, is defined as 
\begin{eqnarray*}
    (T_zF)(z)=zF(z), \text{ }F\in H^2_{\mathbb{R}}(\mathbb{D},\mathbb{C}^m),
\end{eqnarray*}
and its adjoint, denoted by $T^*_z$, called the \textit{backward shift operator}, is given by
\begin{eqnarray*}
    (T^*_zF)(z)=\frac{F(z)-F(0)}{z}, \text{ } F\in H^2_{\mathbb{R}}(\mathbb{D},\mathbb{C}^m).
\end{eqnarray*}   
\end{definition}
 
Throughout this article, we use the term \textit{subspace }for a closed subspace. Suppose $\mathcal{M}$ is a vector space over the field $\mathbb{R}$, then we define a set $\mathcal{M}_\mathbb{C}$ as
\begin{eqnarray*}
 \mathcal{M}_\mathbb{C}:=\big\{f + ih : f,h \in \mathcal{M}\big\}.  
\end{eqnarray*}
 It is a vector space over the field $\mathbb{C}$, called the \textit{complexification of $\mathcal{M}$}. For a vector $A=\begin{pmatrix} a_1 \\ \vdots \\ a_m\end{pmatrix}$ in $\mathbb{C}^m$, we define $\overline{A}:=\begin{pmatrix} \overline{a_1} \\ \vdots \\ \overline{a_m}\end{pmatrix}$.
  For each $F$ in $H^2(\mathbb{D},\mathbb{C}^m)$, we define a function $\widehat{F}: \mathbb{D} \longrightarrow \mathbb{C}^m$ by
\begin{eqnarray*}
    \widehat{F} (z):=\overline{F(\overline{z})}.
\end{eqnarray*}

If $F(z)=\sum_{n=0}^{\infty}A_nz^n \in H^2(\mathbb{D},\mathbb{C}^m)$, then
$
\widehat{F}(z)=\sum_{n=0}^{\infty}\overline{A_n}z^n$. It is easy to see that $\widehat{F} \in H^2(\mathbb{D},\mathbb{C}^m)$.

Given a subspace $\mathcal{N}$ of $H^2(\mathbb{D},\mathbb{C}^m)$, we define $\widehat{\mathcal{N}}:=\{\widehat{F} : F \in \mathcal{N}\}$. It is straightforward to check that $\widehat{\cl N}$ is also a subspace of $H^2(\mathbb{D},\mathbb{C}^m)$. The Hardy spaces $H^2(\mathbb{D}, \mathbb{C}^{m})$ and $H^2(\mathbb{D}) \oplus \cdots \oplus H^2(\mathbb{D})$ ($m$ times), the orthogonal sum of $m$ copies of $H^2(\mathbb{D})$, are unitarily equivalent. The unitary is given by the map 
\begin{center}

$\begin{pmatrix}
a_1 \\ 
a_2\\ 
\vdots \\
a_m
\end{pmatrix}z^n
\mapsto 
\begin{pmatrix}
 a_1 z^n\\ 
a_2 z^n\\
\vdots \\
a_m z^n
\end{pmatrix},$
\end{center}
where $a_1, \dots , a_m$ are complex scalars.
Thus, each $F\in H^2(\mathbb{D},\mathbb{C}^m)$ can be identified as an $m$-tuple $(f_1, \dots , f_m)$, where $f_1, \dots , f_m$ are in $H^2(\mathbb{D})$. Also, $\widehat{F}$ is identified as $(\widehat{f_1}, \dots , \widehat{f_m})$.

\begin{definition}
The function $\phi : H^2(\mathbb{D},\mathbb{C}^m) \longrightarrow H^2(\mathbb{D},\mathbb{C}^m)$ defined by 
\begin{eqnarray*}
    \phi(F)=\frac{F + \widehat{F}}{2}, \text{ } F\in H^2(\mathbb{D},\mathbb{C}^m), 
\end{eqnarray*} 
is called symmetrisation.
\end{definition}

We observe that $\phi$ maps $H^2(\mathbb{D},\mathbb{C}^m)$ onto $H^2_{\mathbb{R}}(\mathbb{D},\mathbb{C}^m)$. The map $\phi$ is a contraction and additive, but not linear. Every function $F$ in $H^2(\mathbb{D},\mathbb{C}^m)$ can be uniquely expressed as $F=F_1 + iF_2$, where $F_1$, $F_2$ are in  $H^2_{\mathbb{R}}(\mathbb{D},\mathbb{C}^m)$. This is easily understood from the following: 
\begin{eqnarray*}
    F=\frac{F + \widehat{F}}{2} + i\bigg(\frac{F - \widehat{F}}{2i}\bigg)=\phi(F) + (F - \phi(F)).
\end{eqnarray*}
For each $F=(f_1, \dots , f_m) \in H^2(\mathbb{D},\mathbb{C}^m)$, $\phi(F)$ equals to $\big(\phi(f_1), \dots , \phi(f_m)\big)$. For a set $\mathcal{K}\subset  H^2(\mathbb{D},\mathbb{C}^m)$, we define $\phi(\mathcal{K}):=\{\phi(F): F \in \mathcal{K}\}$.
We end this section by providing some essential facts, which will be useful in the subsequent sections.

\begin{lemma}\label{Lem1}
Suppose $f$ and $h$ are two functions in $H^2(\mathbb{D})$. Then $$\langle f, h \rangle = \langle \widehat{h}, \widehat{f} \rangle.$$   
\end{lemma}
\begin{proof}
Suppose $f(z)=\sum_{n=0}^{\infty}a_{n}z^n$ and $h(z)=\sum_{n=0}^{\infty}b_{n}z^n$ are in $H^2(\mathbb{D})$. Then $$\langle f, h \rangle =\sum_{n=0}^{\infty} \langle a_n, b_n \rangle_{\mathbb{C}}=\sum_{n=0}^{\infty}\langle\overline{b_n},\text{ }\overline{a_n}\rangle_{\mathbb{C}}=\langle \widehat{h}, \widehat{f} \rangle.$$
\end{proof}

\begin{lemma}\label{Lem2}
Suppose $\psi$ is an inner function in $H^2(\mathbb{D})$, then $$\widehat{H^2(\mathbb{D})\ominus \psi
H^2(\mathbb{D})}=H^2(\mathbb{D}) \ominus \widehat{\psi} H^2(\mathbb{D}).$$
\end{lemma}
\begin{proof}
First, we observe that $\psi$ is inner implies $\widehat{\psi}$ is inner.
Now, suppose $f\in \widehat{H^2(\mathbb{D})\ominus \psi
H^2(\mathbb{D})}$, then $f=\widehat{h}$ for some $h \in H^2(\mathbb{D})\ominus \psi H^2(\mathbb{D})$. Then for any $k \in H^2(\mathbb{D})$, we have $\widehat{f}\in H^2(\mathbb{D})$ and
\begin{eqnarray*}
\langle f, \widehat{\psi} k \rangle=
\langle \widehat{h}, \widehat{\psi} k \rangle= \langle \psi \widehat{k}, h \rangle=0 
\end{eqnarray*}
This shows that $\widehat{H^2(\mathbb{D})\ominus \psi
H^2(\mathbb{D})} \subset H^2(\mathbb{D}) \ominus \widehat{\psi} H^2(\mathbb{D})$.

 Conversely, suppose $f \in H^2(\mathbb{D}) \ominus \widehat{\psi} H^2(\mathbb{D})$, then
\begin{eqnarray*}
\langle \widehat{f}, \psi h \rangle=\langle \widehat{\psi} \widehat{h}, f \rangle=0    \text{  }   \text{ for any } h \in H^2(\mathbb{D}),  
\end{eqnarray*}
which shows that $f\in \widehat{H^2(\mathbb{D})\ominus \psi H^2(\mathbb{D})}$.

Thus $$H^2(\mathbb{D}) \ominus \widehat{\psi} H^2(\mathbb{D})\subset \widehat{H^2(\mathbb{D})\ominus \psi H^2(\mathbb{D})}.$$ This completes the proof.
\end{proof}

\begin{lemma}\label{Lem3}
Suppose $\mathcal{N}$ is a non-zero subspace of $H^2(\mathbb{D})$. Then $$H^2_{\mathbb{R}}(\mathbb{D}) \ominus \phi(\mathcal{N})\subset\phi(H^2(\mathbb{D})\ominus \mathcal{N}).$$
Moreover, if $\widehat{\mathcal{N}}\subset \mathcal{N}$, then equality holds.
\end{lemma}
\begin{proof}
Suppose $f \in H^2_\mathbb{R}(\mathbb{D}) \ominus \phi(\mathcal{N})$. Then $$\langle f, \phi(h) \rangle =0  \text{ 
 } \text{for any } h \in \mathcal{N}.$$ 
This implies that for any $h \in \mathcal{N}$,
\begin{eqnarray*}
\langle f, h \rangle &=& -\langle f, \widehat{h} \rangle\\
&=& -\langle h, \widehat{f} \rangle \ \ \  \text{(by Lemma \ref{Lem1})}\\
&=& -\langle h, f \rangle  \ \ \  (\text{since} \ f =\widehat{f}).  
\end{eqnarray*}
Now, consider the following orthogonal decomposition
$$H^2(\mathbb{D})= \mathcal{N} \oplus (H^2(\mathbb{D})\ominus \mathcal{N}).$$
Then, $f$ can be expressed as $f=f_1 + f_2$, where $f_1 \in \mathcal{N}$ and $f_2 \in H^2(\mathbb{D})\ominus \mathcal{N}$. Then
\begin{eqnarray*}
\langle f_1, f_1 \rangle&=&\langle f_1 + f_2, f_1 \rangle\\
&=&\langle f, f_1 \rangle\\
&=&-\langle f_1, f \rangle\\
&=&-\langle f_1, f_1 \rangle ,   
\end{eqnarray*}
which shows that $\|f_1\|=0$; that is, $f_1=0$. Therefore $$f=\phi(f)=\phi(f_2),$$
which belongs to $\phi(H^2(\mathbb{D})\ominus \mathcal{N})$.

 Further, if $\widehat{\mathcal{N}}\subset \mathcal{N}$. Choose $f \in \phi(H^2(\mathbb{D})\ominus \mathcal{N})$. Then $f=\phi(k)$, for some $k \in H^2(\mathbb{D})\ominus \mathcal{N}$.
Now, for any $\phi(h) \in \phi(\mathcal{N})$,
\begin{eqnarray*}
\langle f, \phi(h) \rangle &=& \langle \phi(k), \phi(h) \rangle\\
&=& \frac{1}{4}\bigg(\langle k,h \rangle +  \langle \widehat{k},h \rangle +  \langle k, \widehat{h} \rangle + \langle \widehat{k},\widehat{h} \rangle \bigg)\\
&=& \frac{1}{4}\bigg(\langle k,h \rangle +  \langle \widehat{h},k \rangle +  \langle k, \widehat{h} \rangle + \langle h,k \rangle \bigg) \quad \text{(By Lemma \ref{Lem1})}\\
&= & 0  \quad\text{(since $\widehat{\mathcal{N}}\subset \mathcal{N}$)}
\end{eqnarray*}
which shows that $\phi(H^2(\mathbb{D})\ominus \mathcal{N}) \subset H^2_{\mathbb{R}}(\mathbb{D})
\ominus \phi(\mathcal{N})$. This completes the proof.
\end{proof}

\section{Nearly invariant subspaces} \label{NIS}
In this section, we shall investigate the description of nearly invariant subspaces of the backward shift operator on the real Hardy space $H^2_\mathbb{R}(\mathbb{D})$.

\begin{definition}\label{defN}
 A subspace $\mathcal{M}$ of $H^2_{\mathbb{R}}(\mathbb{D})$ is called nearly invariant under the backward shift operator $T^*_z$ if $T^*_zf \in \mathcal{M}$, whenever $f \in \mathcal{M}$ and $f(0)=0$. In other words,
 \begin{eqnarray*}
     T^*_z(\mathcal{M}\cap zH^2_{\mathbb{R}}(\mathbb{D}))\subset \mathcal{M}.
 \end{eqnarray*}
 \end{definition}

Before we proceed to the main result of this section (Theorem \ref{ThM1}), we first describe shift invariant subspaces of the real Hardy space $H^2_\mathbb{R}(\mathbb{D})$. Although these subspaces were previously characterized by Mehanna and Singh \cite{MS}, in Proposition \ref{prop1} we recover the same characterization using a different techinique. We start by recalling the famous Beurling's invariant subspace theorem.

\begin{theorem}\label{BTH}
 Any non-zero shift invariant subspace $\mathcal{M}$ of $H^2(\mathbb{D})$ is of the form $\theta H^2(\mathbb{D})$, where $\theta$ is an inner function. Moreover, if $\theta_1$ and $\theta_2$ are two inner functions such that $\theta_1 H^2(\mathbb{D})=\theta_2 H^2(\mathbb{D})$, then there exists a unimodular complex scalar $\lambda$ such that $\theta_1=\lambda \theta_2$.    
\end{theorem}

\begin{proposition}\label{prop1}
Suppose $\mathcal{M}$ is a non-zero shift invariant subspace of $H^2_\mathbb{R}(\mathbb{D})$. Then there exists an inner function $\theta$ in $H^2_\mathbb{R}(\mathbb{D})$ such that $$\mathcal{M}=\theta H^2_\mathbb{R}(\mathbb{D}).$$
\end{proposition}
\begin{proof}
Consider the complexification of $\mathcal{M}$, $$\mathcal{M}_\mathbb{C}:=\mathcal{M}+i\mathcal{M}=\{f+ih : f,h \in \mathcal{M}\}.$$
Then, $\mathcal{M}_\mathbb{C}$ is a non-zero shift invariant subspace of $H^2(\mathbb{D})$. Thus, by Theorem \ref{BTH}, we have
\begin{eqnarray} \label{eqB}
\mathcal{M}_\mathbb{C}=\theta H^2(\mathbb{D}),   \end{eqnarray}
for some inner function $\theta$ in $H^2(\mathbb{D})$. Further, without loss of generality, we can choose $\theta$ in Equation (\ref{eqB}) such that the power series representation of $\theta$ has a non-zero real Fourier coefficient. Indeed, let $\theta=\sum_{n=0}^{\infty}a_nz^n$. Since $\theta$ is non-zero, there exists a non-negative integer $k$ such that $a_k\ne 0$. Let $a_k=e^{i\alpha}|a_k| \text{ for some } \alpha \in \mathbb{R}$. Then $\theta_1:=e^{-i\alpha}\theta$ is an inner function with a non-zero real Fourier coefficient and $\mathcal{M}_\mathbb{C}=\theta_1 H^2(\mathbb{D})$.
 
 Further, we also observe that $\mathcal{M}_\mathbb{C}=\widehat{\mathcal{M}_\mathbb{C}}$. Therefore 
\begin{eqnarray}\label{e3.01}
\theta H^2(\mathbb{D})=\widehat{\theta} H^2(\mathbb{D}).
\end{eqnarray}
Now, since $\theta$ is an inner function, therefore $\widehat{\theta}$ is an inner function. Thus, from Equation (\ref{e3.01}) and Theorem \ref{BTH}, there exists a non-zero unimodular complex scalar $\lambda$ such that $\theta=\lambda \widehat{\theta}$. Further, since $\theta$ has a non-zero real Fourier coefficient, therefore $\lambda=1$, and thus, we get $\theta=\widehat{\theta}$.
Therefore, from Equation (\ref{eqB}), we have 
$$\mathcal{M}=\phi(\mathcal{M}_\mathbb{C})=\phi(\theta H^2(\mathbb{D}))=\theta \phi(H^2(\mathbb{D}))=\theta H^2_\mathbb{R}(\mathbb{D}).$$
This completes the proof.
\end{proof}

\begin{theorem}\label{ThM1}
Suppose $\mathcal{M}$ is a non-zero subspace of $H_\mathbb{R}^2(\mathbb{D})$ which is nearly invariant under the backward shift operator $T^*_z$. Then there exists a $T^*_z$-invariant subspace $\mathcal{N}$ of $H_\mathbb{R}^2(\mathbb{D})$, and a function $g$ in $\mathcal{M}$, orthogonal to $\mathcal{M}\cap zH^2_\mathbb{R}(\mathbb{D})$, satisfying $g(0)>0$ such that $$\mathcal{M}=g\mathcal{N}.$$
Further, for all $gf \in \mathcal{M}$,  $\|gf\|=\|f\|$. 
\end{theorem}
\begin{proof}
We consider the complexification of $\mathcal{M}$, $$\mathcal{M}_\mathbb{C}:=\mathcal{M}+i\mathcal{M}=\{f+ih : f,h \in \mathcal{M}\}.$$
Then we observe that $\mathcal{M}_\mathbb{C}$ is a non-zero subspace of $H^2(\mathbb{D})$. Further,
since $\mathcal{M}$ is a nearly $T^*_z$-invariant subspace of $H^2_{\mathbb{R}}(\mathbb{D})$, therefore, $\mathcal{M}_\mathbb{C}$ is a nearly  $T^*_z$-invariant subspace of $H^2(\mathbb{D})$. Thus, by Theorem \ref{ThH}, there exists a $T^*_z$-invariant subspace $\mathcal{K} \subset H^2(\mathbb{D})$, and a function $g \in \mathcal{M}_{\mathbb{C}}$, orthogonal to $\mathcal{M}_\mathbb{C}\cap zH^2(\mathbb{D})$ and $g(0)>0$, such that $$\mathcal{M}_\mathbb{C}=g\mathcal{K}.$$
 Since, $\mathcal{M}_\mathbb{C}=\widehat{\mathcal{M}_\mathbb{C}}$ and $g \in \mathcal{M}_\mathbb{C}$, therefore $\widehat{g} \in \mathcal{M}_\mathbb{C}$. Also, for any $h=zf \in \mathcal{M}_\mathbb{C}\cap zH^2(\mathbb{D})$
 \begin{eqnarray*}
 \langle \widehat{g},h \rangle =\langle \widehat{g},zf \rangle = \langle z\widehat{f}, g \rangle = 0. &&\text{(by Lemma \ref{Lem1})}   
 \end{eqnarray*}
This shows that $\widehat{g}$ is orthogonal to  $\mathcal{M}_\mathbb{C}\cap zH^2(\mathbb{D})$. Now since, $\mathcal{M}_{\mathbb{C}}$ is a nearly $T^*_z$-invariant subspace of $H^2(\mathbb{D})$, therefore $\mathcal{M}_\mathbb{C}\ominus (\mathcal{M}_\mathbb{C}\cap zH^2(\mathbb{D})) $ is one dimensional. Therefore $\widehat{g}=\alpha g$, for some $\alpha \in \mathbb{C}$. Since $g(0)>0$, thus $\widehat{g}(0)=g(0)$ and hence $\alpha=1$. This concludes that $g=\widehat{g}$. Thus, $g\in \mathcal{M} \ominus (\mathcal{M}\cap zH^2_{\mathbb{R}}(\mathbb{D}))$. 

 Further, $\mathcal{M}_\mathbb{C}=g\mathcal{K}$ implies 
\begin{eqnarray*}
\mathcal{M}&=&\phi(\mathcal{M}_\mathbb{C})\\
&=&\phi(g\mathcal{K})\\&=&g\phi(\mathcal{K})\\&=&g\mathcal{N}, \text{ where } \mathcal{N}=\phi(\mathcal{K}).
\end{eqnarray*}
Now we show that $\mathcal{N}$ is a closed $T^*_z$-invariant subspace. Since $\mathcal{K}$ is a vector subspace of $H^2(\mathbb{D})$ and $\mathcal{N}=\phi(\mathcal{K})$, therefore it is easy to see that $\mathcal{N}$ is a vector subspace of $H^2_{\mathbb{R}}(\mathbb{D})$. Now, take a sequence $\{h_n\}_{n=1}^{\infty}$ in $\mathcal{N}$  such that $h_n \longrightarrow h$ for some $h$ in $H^2_{\mathbb{R}}(\mathbb{D})$, then $h_n=\phi(f_n)$ for $f_n \in \mathcal{K}$. Since $g=\widehat{g}$, therefore $\mathcal{K}=\widehat{\mathcal{K}}$, and thus $\phi(\mathcal{K}) \subset \mathcal{K}$. Hence $h \in \mathcal{K}$. Moreover,
\begin{eqnarray*}
\|gh_n-gh\|&=&\|g\phi(f_n)-gh\|\\ &=&\|\phi(f_n)-h\|\longrightarrow 0, \text{ as }n \rightarrow \infty. \quad \text{(by Theorem \ref{ThH} )}    
\end{eqnarray*}
Now, since $\mathcal{M}$ is closed and $gh_n\in \mathcal{M}$. Therefore $gh \in \mathcal{M}=g\mathcal{N},$
which shows that $h \in \mathcal{N}$. Thus $\mathcal{N}$ is closed.

 Next, take an element $h\in \mathcal{N}$. Then $h=\phi(f)$ for some $f\in \mathcal{K}$. Then
\begin{eqnarray*}
T^*_zh&=&
T^*_z\phi(f)\\&=&T^*_z\bigg(\frac{f + \widehat{f}}{2}\bigg)\\
&=&\frac{T^*_zf + \widehat{T^*_zf}}{2} \quad \text{(since $T^*_z\widehat{f}=\widehat{T^*_zf})$}\\
&\in &  \phi(\mathcal{K})=\mathcal{N},    
\end{eqnarray*}
which shows that $\mathcal{N}$ is $T^*_z$-invariant.

 Lastly, since $\phi(\mathcal{K})\subset \mathcal{K}$, therefore, for any $h=gf\in \mathcal{M}$, where $f \in \mathcal{N}= \phi(\mathcal{K})$,
\begin{eqnarray*}
\|h\|=\|gf\|=\|f\|. &&\text{(by Theorem \ref{ThH})}  
\end{eqnarray*}
This completes the proof.   
\end{proof}

\textnormal{We end this section with a few remarks that illuminate some of the structural properties of the spaces $\mathcal{N}$ and $\mathcal{K}$, which appeared in the proof of Theorem \ref{ThM1}.}

\begin{remark}
 We show that $\mathcal{K}=\mathcal{N}_\mathbb{C}$, where $\mathcal{N}_\mathbb{C}$ is the complexification of $\mathcal{N}$. This is easily understood because $\mathcal{M}_\mathbb{C}=g\mathcal{K}$ and $\mathcal{M}=g\mathcal{N}$. Therefore, $\mathcal{M}_\mathbb{C}=\mathcal{M}+i\mathcal{M}
 =g\mathcal{N}_\mathbb{C}$,
 which implies $\mathcal{K}=\mathcal{N}_\mathbb{C}$.

\end{remark}

\begin{remark}\label{Rem3.3}
In Theorem \ref{ThM1}, we observed that $\mathcal{K}$ and $\mathcal{N}$ are backward shift invariant subspaces of $H^2(\mathbb{D})$ and $H^2_{\mathbb{R}}(\mathbb{D})$, respectively.

Now, since $\mathcal{K}$ is a $T^*_z$-invariant subspace of $H^2(\mathbb{D})$. Therefore, by Theorem \ref{BTH},
\begin{eqnarray*} \label{e3.1}
\mathcal{K}=H^2(\mathbb{D})\ominus\theta H^2(\mathbb{D}),   
\end{eqnarray*}
for some inner function $\theta$ in $H^2(\mathbb{D})$. Here, we can choose  $\theta$ such that the power series representation of $\theta$ has a non-zero real Fourier coefficient.

 Also, $\mathcal{N}$ is a $T^*_z$-invariant susbspace of $H^2_{\mathbb{R}}(\mathbb{D})$. Therefore, by using Proposition \ref{prop1},
$$\mathcal{N}=H^2_{\mathbb{R}}(\mathbb{D})\ominus \psi H^2_{\mathbb{R}}(\mathbb{D}),$$
for some inner function $\psi$ in $H^2_{\mathbb{R}}(\mathbb{D})$.

We establish the connection between these two inner functions. More specifically, we show that $\theta=\psi$. This is as follows:

Since $\mathcal{K}=\widehat{\mathcal{K}}$ and by using Lemma \ref{Lem2}, we have
\begin{eqnarray*}
H^2(\mathbb{D}) \ominus \theta H^2(\mathbb{D})&=&\widehat{H^2(\mathbb{D}) \ominus \theta H^2(\mathbb{D})}\\&=&H^2(\mathbb{D}) \ominus \widehat{\theta} H^2(\mathbb{D}),    
\end{eqnarray*}
which shows that $ \theta H^2(\mathbb{D})= \widehat{\theta} H^2(\mathbb{D})$. This, together with the fact that $\theta$ has a non-zero Fourier coefficient, implies that $\theta=\widehat{\theta}$.
Further, using Lemma \ref{Lem3}, we get $$H^2_{\mathbb{R}}(\mathbb{D}) \ominus \phi(\mathcal{K})=\phi(H^2(\mathbb{D})\ominus \mathcal{K}),$$
which implies $$\psi H^2_{\mathbb{R}}(\mathbb{D})=\phi(\theta H^2(\mathbb{D}))=\theta H^2_{\mathbb{R}}(\mathbb{D}).$$
Consequently, we get $\theta H^2(\mathbb{D})=\psi H^2(\mathbb{D})$.
Thus $\theta=\lambda \psi$ for some unimodular complex scalar $\lambda$. Since $\theta=\widehat{\theta}$ and $\psi=\widehat{\psi}$, therefore we get,
\begin{eqnarray*}
   \lambda \psi= \overline{\lambda}\psi,
\end{eqnarray*}
which shows that $\lambda=\overline{\lambda}$. Therefore $\lambda=1$, and hence $\theta=\psi$.

\end{remark}

\textnormal{The following remark concerns Hitts's theorem (Theorem \ref{ThH}) if $\mathcal{M}$ is a nearly $T^*_z$-invariant subspace of $H^2(\mathbb{D})$ and $\mathcal{M}=\widehat{\mathcal{M}}$. Then the unique function $g$ given by Hitt's theorem  in the description of $\mathcal{M}$ satisfies $g=\widehat{g}$.}
\begin{remark}\label{Rem3.5}
Consider a subspace $\mathcal{M}$ of $H^2(\mathbb{D})$ that is nearly $T^*_z$-invariant and satisfies $\mathcal{M}=\widehat{\mathcal{M}}$. Then we observe that $\phi(\mathcal{M})=\mathcal{M} \cap H^2_{\mathbb{R}}(\mathbb{D})$ and  $\phi(\mathcal{M})$ is a nearly $T^*_z$-invariant subspace of $H^2_{\mathbb{R}}(\mathbb{D})$. Therefore, by Theorem \ref{ThM1}, there exists $g \in \phi(\mathcal{M})$ satisfying $g=\widehat{g}$ such that $$\phi(\mathcal{M})=g\mathcal{N},$$
where $\mathcal{N}$ is a $T^*_z$-invariant subspace of $H^2_{\mathbb{R}}(\mathbb{D})$. Further, the condition $\mathcal{M}=\widehat{\mathcal{M}}$ implies that $\mathcal{M}=\phi(\mathcal{M}) + i\phi(\mathcal{M})$. Indeed, $\phi(\mathcal{M}) + i\phi(\mathcal{M}) \subset \mathcal{M}$, and for any $f \in \mathcal{M},$
\begin{eqnarray*}
 f=\frac{f + \widehat{f}}{2} + i\Bigg(\frac{f - \widehat{f}}{2i}\Bigg) \in \phi(\mathcal{M}) + i\phi(\mathcal{M}). 
\end{eqnarray*}
This concludes that
\begin{align*}
 \mathcal{M}= \phi(\mathcal{M}) + i\phi(\mathcal{M})=g\mathcal{N}_{\mathbb{C}}, 
\end{align*} 
where $\mathcal{N_{\mathbb{C}}}=\mathcal{N} + i\mathcal{N}$ is a $T^*_z$-invariant subspace of $H^2(\mathbb{D})$. This is what we required.
\end{remark}

\section{Nearly invariant subspaces with finite defect} \label{NISF}

\textnormal{In this section, we shall describe nearly invariant subspaces with finite defect under the backward shift operator on the real Hardy space $H^2_{\mathbb{R}}(\mathbb{D})$.  }

\begin{definition}\label{defD}
 A subspace $\mathcal{M}$ of $H^2_{\mathbb{R}}(\mathbb{D})$ is called nearly invariant with defect under the backward shift operator $T^*_z$ if there exists a finite dimensional subspace $\mathcal{F}$ (orthogonal to $\mathcal{M}$) of $H^2_{\mathbb{R}}(\mathbb{D})$ such that $T^*_zf \in \mathcal{M} \oplus \mathcal{F}$, whenever $f \in \mathcal{M}$ and $f(0)=0$. In other words, 
 \begin{eqnarray} \label{eD}
     T^*_z(\mathcal{M}\cap zH^2_{\mathbb{R}}(\mathbb{D}))\subset \mathcal{M} \oplus \mathcal{F}.
 \end{eqnarray}
A finite dimensional subspace $\mathcal{F}$ of the smallest dimension that satisfies (\ref{eD}) is called the \textit{defect space} of $\mathcal{M}$, and its dimension is called the \textit{defect} of $\mathcal{M}$. 
\end{definition}

\textnormal{Before we present the main theorem of this section (Theorem \ref{MT2}), we recall the following result due to Chalendar, Gallardo-Guti$\rm{\acute{e}}$rrez, and Partington from \cite{CGP1}.}

\begin{theorem}\cite[Theorem 3.2]{CGP1}\label{CGP}
Suppose $\mathcal{M}$ is a non-zero subspace of $H^2(\mathbb{D})$ which is nearly $T^*_z$-invariant with defect $n$, and let $\{e_i\}_{i=1}^{n}$ be an orthonormal basis of the defect space.
\begin{itemize}
    \item[(i)] If there exists a function in $\mathcal{M}$ which does not vanish at $0$, then there exists a $T^*_z\oplus T^*_z$-invariant subspace $\mathcal{K}\subset H^2(\mathbb{D}) \oplus H^2(\mathbb{D},\mathbb{C}^n)$ such that
\begin{eqnarray*}
    \mathcal{M}=\Big\{f=gk + z\sum_{i=1}^{n}k_ie_i : (k,k_1,\dots , k_n)\in \mathcal{K}\Big\},
\end{eqnarray*}
where $g$ is the function of unit norm in $\mathcal{M}$, orthogonal to $\mathcal{M}\cap zH^2(\mathbb{D})$, satisfying $g(0)>0$. Further, for all $f\in \mathcal{M}$,
\begin{eqnarray*}
\|f\|^2=\|k\|^2 +\sum_{i=1}^{n} \|k_i\|^2.
\end{eqnarray*}

\item[(ii)] If all functions in $\mathcal{M}$ vanish at $0$, then there exists a $T^*_z$-invariant subspace $\mathcal{K}\subset H^2(\mathbb{D},\mathbb{C}^n)$ such that
$$\mathcal{M}=\Big\{f=z\sum_{i=1}^{n}k_ie_i : (k_1,\dots, k_n)\in \mathcal{K}\Big\}.$$
 Further, for all $f \in \mathcal{M}$,
\begin{eqnarray*}
    \|f\|^2=\sum_{i=1}^{n}\|k_i\|^2.
\end{eqnarray*}
\end{itemize}

\end{theorem}

\textnormal{For the sake of notational clarity and simplicity, we restrict our attention to the case of defect one, focusing on the essential features of the result while omitting lengthy technical details. The extension to the general case follows in a similar manner.}

\begin{theorem}\label{MT2}
Suppose $\mathcal{M}$ is a non-zero subspace of $H^2_{\mathbb{R}}(\mathbb{D})$ which is nearly $T^*_z$-invariant with defect 1, and let the defect space be spanned by a unit vector $e$. 
\begin{itemize}
    \item[(i)] If there exists a function in $\mathcal{M}$ which does not vanish at $0$, then there exists a $T^*_z\oplus T^*_z$-invariant subspace $\mathcal{N}\subset H^2_{\mathbb{R}}(\mathbb{D}) \oplus H^2_\mathbb{R}(\mathbb{D})$ such that
\begin{eqnarray*}
    \mathcal{M}=\Big\{f=gh + zh_1e : (h, h_1)\in \mathcal{N}\Big\},
\end{eqnarray*}
where $g$ is the function of unit norm in $\mathcal{M}$, orthogonal to $\mathcal{M}\cap zH^2_{\mathbb{R}}(\mathbb{D})$, satisfying $g(0)>0$. Further, for all $f\in \mathcal{M}$,
\begin{eqnarray*}
\|f\|^2=\|h\|^2 + \|h_1\|^2.
\end{eqnarray*}
\item[(ii)] If all functions in $\mathcal{M}$ vanish at $0$, then there exists a $T^*_z$-invariant subspace $\mathcal{N}\subset H^2_{\mathbb{R}}(\mathbb{D})$ such that
$$\mathcal{M}=\Big\{f=zh_1e : h_1\in \mathcal{N}\Big\}.$$
 Further, for all $f \in \mathcal{M}$,
\begin{eqnarray*}
    \|f\|=\|h_1\|.
\end{eqnarray*}
\end{itemize}
\end{theorem}
\begin{proof}
Consider the complexification of $\mathcal{M}$, 
$$\mathcal{M}_\mathbb{C}:=\mathcal{M}+i\mathcal{M}=\{f+ih : f,h \in \mathcal{M}\}.$$
Then we observe that $\mathcal{M}_{\mathbb{C}}$ is a nearly $T^*_z$-invariant subspace of $H^2(\mathbb{D})$ with defect 1, and the defect space is spanned by the unit vector $e$. 

  Now, first assume that not all functions in $\mathcal{M}$ vanish at $0$. This implies that not all functions in $\mathcal{M}_{\mathbb{C}}$ vanish at $0$. Therefore, by Theorem \ref{CGP} for the case $n=1$, there exists a $T^*_z\oplus T^*_z$-invariant subspace $\mathcal{K}\subset H^2(\mathbb{D})\oplus H^2(\mathbb{D})$ such that
  $$\mathcal{M}_{\mathbb{C}}=\Big\{l=gk + zk_1e : (k,k_1)\in \mathcal{K}\Big\},$$
  where $g$ is the function of unit norm in $\mathcal{M}_\mathbb{C}$, orthogonal to $\mathcal{M}_\mathbb{C}\cap zH^2(\mathbb{D})$, satisfying $g(0)>0$. Since, $\mathcal{M}_\mathbb{C}=\widehat{\mathcal{M}_\mathbb{C}}$ and $g \in \mathcal{M}_\mathbb{C}$, therefore $\widehat{g} \in \mathcal{M}_\mathbb{C}$. Also, for any $l=zt \in \mathcal{M}_\mathbb{C}\cap zH^2(\mathbb{D})$,
 \begin{eqnarray*}
 \langle \widehat{g},l \rangle =\langle \widehat{g},zt \rangle = \langle z\widehat{t}, g \rangle = 0. &&\text{(By Lemma \ref{Lem1})}   
 \end{eqnarray*}
 This shows that $\widehat{g}$ is orthogonal to  $\mathcal{M}_\mathbb{C}\cap zH^2(\mathbb{D})$. Now, since $\mathcal{M}_{\mathbb{C}}$ is not contained in $zH^2(\mathbb{D})$, therefore $\mathcal{M}_\mathbb{C}\ominus (\mathcal{M}_\mathbb{C}\cap zH^2(\mathbb{D})) $ is one dimensional. Therefore $\widehat{g}=\alpha g$, for some $\alpha \in \mathbb{C}$. Since $g(0)>0$, thus $\widehat{g}(0)=g(0)$ and hence $\alpha=1$. This concludes that $g=\widehat{g}$. Thus, $g\in \mathcal{M} \ominus (\mathcal{M}\cap zH^2_{\mathbb{R}}(\mathbb{D}))$. 
Thus we have 
\begin{eqnarray*}
    \mathcal{M} &=& \phi(\mathcal{M}_\mathbb{C})\\
    &=&\{\phi(l): l \in \mathcal{M}_\mathbb{C} \}\\
    &=& \Big\{f=g\phi(k) + z\phi(k_1)e : (k,k_1) \in \mathcal{K}\Big\}\\
    &=& \Big\{f=gh + zh_1e : h=\phi(k), h_1=\phi(k_1)\Big\}.
\end{eqnarray*}
Now consider the set $\mathcal{N}\subset H^2_{\mathbb{R}}(\mathbb{D}) \oplus H^2_{\mathbb{R}}(\mathbb{D})$ defined by
\begin{eqnarray*}
    \mathcal{N}:=\{(h,h_1): f=gh+zh_1e \in \mathcal{M}\}.
\end{eqnarray*}
This shows that
$\mathcal{N}=\phi(\mathcal{K})$. Thus, for any $(k, k_1) \in \mathcal{K}$, $$l=gk + zk_1e  \in \mathcal{M}_{\mathbb{C}},$$
and since $\mathcal{M}_{\mathbb{C}}=\widehat{\mathcal{M}}_{\mathbb{C}}$, therefore $$\widehat{l}=g\widehat{k} + z\widehat{k}_1e  \in \mathcal{M}_{\mathbb{C}},$$
which implies $(\widehat{k},\widehat{k}_1)\in \mathcal{K}$.
This shows that $\widehat{\mathcal{K}}\subset \mathcal{K}$. Therefore, $$\mathcal{N}=\phi(\mathcal{K})\subset \mathcal{K}.$$ This leads us to have 
\begin{eqnarray*}
    \|f\|^2 &=& \|gh + h_1\|^2\\
    &=& \|h\|^2 + \|h_1\|^2. \quad \text{(by Theorem \ref{CGP})}\\
\end{eqnarray*}
Thus we conclude that any $f\in \mathcal{M}$ can be expressed as $$ f = gh + zh_1e,$$
together with norm equality
\begin{eqnarray*}
    \|f\|^2=\|h\|^2+\|h_1\|^2.
\end{eqnarray*}
Next, we show that $\mathcal{N}$ is a closed subspace of $H^2_{\mathbb{R}}(\mathbb{D}) \oplus H^2_{\mathbb{R}}(\mathbb{D})$. Since $\mathcal{K}$ is a vector subspace of $H^2(\mathbb{D}) \oplus H^2(\mathbb{D})$ and $\mathcal{N}=\phi(\mathcal{K})$, therefore it is easy to see that $\mathcal{N}$ is a vector subspace of $H^2_{\mathbb{R}}(\mathbb{D}) \oplus H^2_{\mathbb{R}}(\mathbb{D})$.
We show that $\mathcal{N}$ is closed. Take a Cauchy sequence $\{(h^n,h^n_1)\}_{n=1}^{\infty}$ in $\mathcal{N}$. Then $$f_n=gh^n + zh^n_1e \in \mathcal{M}$$
and   $$\|f_n\|^2=\|h^n\|^2 + \|h^n_1\|^2,$$ which is equal to $\|(h^n,h^n_1)\|^2$. This shows that the sequence $\{f_n\}$ is Cauchy and thus it converges to a point $f \in \mathcal{M}$. Therefore there exists $(h,h_1) \in \mathcal{K}$ such that $f=gh + zh_1e$. Thus 

\begin{eqnarray*}
\|(h^n,h^n_1)-(h,h_1)\|^2
&=&\|(h^n -h, h^n_1-h_1)\|^2\\
&=&\|(h^n-h)\|^2 + \|(h^n_1-h_1)\|^2\\
&=&\|f_n-f\|^2 \longrightarrow 0 \text{ as } n \longrightarrow \infty.
\end{eqnarray*}
This shows that $\mathcal{N}$ is closed in $H^2_{\mathbb{R}}
(\mathbb{D}) \oplus H_\mathbb{R}^2(\mathbb{D})$.

Lastly, we show that $\mathcal{N}$ is invariant under $T^*_z\oplus T^*_z$. Suppose $(h,h_1)\in \mathcal{N}$, then there exists $(k, k_1)\in \mathcal{K}$ such that $h=\phi(k), h_1=\phi(k_1)$. Thus 
\begin{eqnarray*}
    (T^*_z\oplus T^*_z) (h,h_1) &=& (T^*_z\oplus T^*_z)(\phi(k),\phi(k_1))\\
    &=& (T^*_z\phi(k),T^*_z\phi(k_1))\\
    &=& \bigg(T^*_z\bigg(\frac{k+\widehat{k}}{2}\bigg), T^*_z\bigg(\frac{k_1+\widehat{k_1}}{2}\bigg) \bigg)\\
    &=& \bigg(\frac{T^*_zk+T^*_z\widehat{k}}{2}, \frac{T^*_zk_1+T^*_z\widehat{k_1}}{2}\bigg) \\
    &=& \bigg(\frac{T^*_zk+\widehat{T^*_zk}}{2}, \frac{T^*_zk_1+\widehat{T^*_zk_1}}{2}\bigg) \quad \text{( since $T^*_z\widehat{k}=\widehat{T^*_zk}$, $T^*_z\widehat{k}_1=\widehat{T^*_zk_1}$)}\\
    &=& (\phi(T^*_zk), \phi(T^*_zk_1))\\
    &\in &\phi(\mathcal{K})=\mathcal{N},
\end{eqnarray*}
which is what we required.

 Further, for the other case, when all functions in $\mathcal{M}$ vanish at $0$, then this shows that all the functions in $\mathcal{M}_{\mathbb{C}}$ vanish at $0$. Therefore, by following similar techniques to those above, we get that 
 
 $$\mathcal{M}=\Big\{f=zh_1e : h_1\in \mathcal{N}\Big\},$$
where the subspace
\begin{eqnarray*}
    \mathcal{N}=\Big\{h_1: f=zh_1e\in \mathcal{M}\Big\}
\end{eqnarray*}
is $T^*_z$-invariant in $H^2_{\mathbb{R}}(\mathbb{D})$. Further, for all $f=zh_1e \in \mathcal{M}$,
\begin{eqnarray*}
    \|f\|=\|h_1\|.
\end{eqnarray*}
This completes the proof.
\end{proof}

\textnormal{Using arguments similar to those employed in the proof of Theorem \ref{MT2}, we can extend Theorem \ref{MT2} to the case of any finite defect. For the sake of completeness, we state the result below.}

\begin{theorem}\label{MT3}
Suppose $\mathcal{M}$ is a non-zero subspace of $H^2_{\mathbb{R}}(\mathbb{D})$ which is nearly $T^*_z$-invariant with defect $n$, and let $\{e_i\}_{i=1}^{n}$ be an orthonormal basis of the defect space. 
\begin{itemize}
    \item[(i)] If there exists a function in $\mathcal{M}$ which does not vanish at $0$, then there exists a $T^*_z\oplus T^*_z$-invariant subspace $\mathcal{N}$ of $H^2_{\mathbb{R}}(\mathbb{D}) \oplus H^2_{\mathbb{R}}(\mathbb{D},\mathbb{C}^n)$ such that
\begin{eqnarray*}
    \mathcal{M}=\Bigg\{f=gh + z\sum_{i=1}^{n}h_ie_i: (h, h_1, \dots, h_n) \in \mathcal{N}\Bigg\},
\end{eqnarray*}
where $g$ is the function of unit norm in $\mathcal{M}$, orthogonal to $\mathcal{M}\cap zH^2_{\mathbb{R}}(\mathbb{D})$, satisfying $g(0)>0$.
Further, for all $f \in \mathcal{M}$,
\begin{eqnarray*}
\|f\|^2=\|h\|^2 + \sum_{i=1}^{n}\|h_i\|^2.
\end{eqnarray*}
\item[(ii)] If all functions in $\mathcal{M}$ vanish at $0$, then there exists a $T^*_z$-invariant subspace $\mathcal{N}$ of $H^2_{\mathbb{R}}(\mathbb{D},\mathbb{C}^n)$ such that
$$\mathcal{M}=\Bigg\{f=z\sum_{i=1}^{n}h_ie_i :(h_1, \dots, h_n) \in \mathcal{N}\Bigg\}.$$
Further, for all $f\in \mathcal{M},$
\begin{eqnarray*}
    \|f\|^2=\sum_{i=1}^{n}\|h_i\|^2.
\end{eqnarray*}
\end{itemize}
\end{theorem}

\textnormal{Now we shall describe almost invariant subspaces of the backward shift operator on the real Hardy space:}

\begin{definition}\label{defA}
 A subspace $\mathcal{M}$ of $H^2_{\mathbb{R}}(\mathbb{D})$ is called almost invariant under the backward shift operator $T^*_z$ if there exists a finite dimensional subspace $\mathcal{F}$ (orthogonal to $\mathcal{M}$) of $H^2_{\mathbb{R}}(\mathbb{D})$ such that 
 \begin{eqnarray}\label{eA}
  T^*_z\mathcal{M} \subset \mathcal{M} \oplus \mathcal{F}.  
 \end{eqnarray}
A finite dimensional subspace $\mathcal{F}$ of the smallest dimension that satisfies (\ref{eA}) is called the \textit{defect space} of $\mathcal{M}$, and its dimension is called the \textit{defect} of $\mathcal{M}$. 
\end{definition}

\textnormal{Interestingly, almost $T^*_z$-invariant subspaces are linked with nearly $T^*_z$-invariant subspaces with finite defect in a natural way.
In the following corollary, we utilize Theorem \ref{MT3} to describe almost $T^*_z$-invariant subspaces in the real Hardy space $H^2_{\mathbb{R}}(\mathbb{D})$:}
\begin{corollary}\label{CorA}
A subspace $\mathcal{M}$ of $H^2_{\mathbb{R}}(\mathbb{D})$ is almost invariant under the backward shift operator $T^*_z$ with finite defect $n$ if and only if $\mathcal{M}$ is of the form described in either case (i) or case (ii) of Theorem \ref{MT3}, with an additional condition that $T^*_zg \in \mathcal{M} \oplus \mathcal{F}$ ($\mathcal{F}$ is the defect space) in case $(i)$, while case $(ii)$ is unchanged.  
\end{corollary}
\begin{proof}
 Suppose $\mathcal{M}$ is almost invariant under $T^*_z$ with defect $n$. Then $\mathcal{M}$ is nearly $T^*_z$-invariant with defect $n$. Thus, $\mathcal{M}$ satisfies the hypothesis of Theorem \ref{MT3}. Further, since $g \in \mathcal{M}$, therefore $T^*_zg \in \mathcal{M} \oplus \mathcal{F}$ also holds.

  Conversely, suppose $\mathcal{M}$ satisfies the hypothesis of Theorem \ref{MT3} together with $T^*_zg \in \mathcal{M} \oplus \mathcal{F}$. Then for any $f \in \mathcal{M}$
\begin{eqnarray*}
     f=gh + z\sum_{i=1}^{n} h_i e_i=gh(0) + g(h - h(0)) + z\sum_{i=1}^{n} h_i e_i.
\end{eqnarray*}
Set $l(z)=g(h - h(0)) + z\sum_{i=1}^{n} h_i e_i$. Then we observe that $l \in \mathcal{M}$ and since $l(0)=0$, therefore $T^*_zl \in \mathcal{M} \oplus \mathcal{F}$. This, along with the fact that $T^*_zg \in \mathcal{M} \oplus \mathcal{F}$ concludes that $T^*_zf \in\mathcal{M} \oplus \mathcal{F}$. This completes the proof.

\end{proof}

\subsection*{Acknowledgements}
We thank the Mathematical Sciences Foundation, Delhi, for the support and facilities needed to complete the present work. The first author thanks the University Grants Commission(UGC), India, for the support, and the second author thanks Shiv Nadar Institution of Eminence for partially supporting this research.

\subsection*{Statements and Declarations} 

\vspace{.2 cm}

\noindent {\bf Competing Interests}  The authors have no competing interests to declare that are relevant to the content of this article.  

\vspace{.2 cm}

\noindent {\bf Data availability statements} Data sharing is not applicable to this article as no new data were created or analyzed in this study.

\bibliography{refs.bib}

\end{document}